\documentclass[12pt]{amsart}
\usepackage{amsmath}
\usepackage{amssymb}
\usepackage{amsfonts}
\usepackage{amsthm}
\usepackage{xcolor}
\usepackage{enumerate}

\newtheorem{theo}{Theorem}[section]
\newtheorem*{theo*}{Theorem}
\newtheorem{coro}[theo]{Corollary}
\newtheorem{lemm}[theo]{Lemma}

\newcommand{\N}{\mathbb{N}}
\newcommand{\Z}{\mathbb{Z}}

\title[Menon-type identity]{A Menon-type identity using Klee's function }
\begin{document}
 
 \author[A Chandran]{Arya Chandran}
 \address{Department of Mathematics, University College, Thiruvananthapuram, Kerala - 695034, India}
 \email{aryavinayachandran@gmail.com}
 \author[N E Thomas]{Neha Elizabeth Thomas}
 \address{Department of Mathematics, SD College, Alappuzha, Kerala - 688003, India}
 \email{nehathomas2009@gmail.com}
\author[K V Namboothiri]{K Vishnu Namboothiri}
\address{Department of Mathematics, Government College, Ambalapuzha, Kerala - 688561, INDIA\\Department of Collegiate Education, Government of Kerala, India}
\email{kvnamboothiri@gmail.com}

 \begin{abstract}
  Menon's identity is a classical identity involving gcd sums and the Euler totient function $\phi$. A natural generalization of $\phi$ is the Klee's function $\Phi_s$. In this paper we derive  a Menon-type identity using Klee's function and a generalization of the gcd function. This identity generalizes an identity given by Lee and Kim in [\textit{J. Number Theory 175, 42--50(2017)}].
 \end{abstract}
 
 \maketitle
\section{Introduction}
The Euler totient function $\phi$ appears in many interesting identities in number theory. Probably because of its applications in various branches of number theory, it has been generalized in many ways. The Jordan function $J_s(n)$, the von Sterneck's function $H_s(n)$, the Cohen's function $\phi_s$ \cite{cohen1956some} and the Klee's function $\Phi_s$ \cite{klee1948generalization} are some important extensions of $\phi$ (see definitions in the next section). All these functions share several common properties. For example, Euler totient function $\phi$ holds a relation with the M\"{o}bius function. Similar relations are satisfied by all these generalizations.  All these generalizations have a product formulae in terms of the prime factorization of their  arguments. Hence all these are multiplicative and behave similar to  $\phi$ on prime powers.

Note that Cohen has proved in \cite{cohen1956some} the equality of the functions $J_s$, $H_s$ and $\phi_s$ though they are all defined differently. Klee's function $\Phi_s$ and Cohen's $\phi_s$ are connected by the relation $\phi_s(n) = \Phi_s(n^s)$. Thus $\Phi_s$ seems to be a natural generalization of $\phi$ (as for $s=1$, the former turns out to be equal to the latter). 

The classical Menon's identity which originally appeared in \cite{menon1965sum} is a gcd sum turning out to be equivalent to a product of the Euler function and the number of divisors function $\tau$. If $(m,n)$ denotes the gcd of $m$ and $n$, the identity is precisely the following:
\begin{align}\label{menons-identity}
\sum\limits_{\substack{m=1\\(m,n)=1}}^n (m-1,n)=\phi(n)\tau(n).
\end{align}
It has been generalized and extended by many authors. Many of the identities were derived using elementary number theory techniques. For example, in a recent paper, Zhao and Kao \cite{zhao2017another} suggested a generalization involving Dirichlet characters mod $n$ using elementary number theoretic methods. Their identity is
\begin{align*}
\sum\limits_{\substack{m=1\\(m,n)=1}}^n (m-1,n)\chi(m)=\phi(n)\tau\left(\frac{n}{d}\right),
\end{align*} 
where $\chi$ is a Dirichlet character mod $n$ with conductor $d$. When one takes $\chi$ as the principal character mod $n$, this identity turns to be equal to the Menon's identity. After this, a similar type of identity in terms of even functions mod $n$ was given by L. T{\'o}th \cite{toth2018menon}.  An arithmetical function $f$ is $n-$even (or even mod $n$) if $f(r)=f((r,n))$. T{\'o}th also used elementary number theory techniques and properties of arithmetical functions to prove his identity. K. N. Rao in \cite{rao1972certain} gave a generalization of the form
\begin{align*}
\sum\limits_{\substack{m_i \in U_k(n)}}(m_1-s_1, m_2-s_2\ldots,m_k-s_k,n)^k= J_k(n) \tau(n),
\end{align*}
where $U_k(n)$ is the unit group modulo $n$ and $s_i\in\Z$. He used Cauchy composition and finite Fourier representations to establish this result.

A different approach was used by B. Sury in \cite{sury2009some}. He used the method of group actions to derive the following identity.
\begin{align}\label{sury-identity}
\sum\limits_{\substack{1\leq m_1, m_2,\ldots,m_k\leq n\\(m_1,n)=1}}(m_1-1, m_2,\ldots,m_k,n)=\phi(n)\sigma_{k-1}(n),
\end{align} where $\sigma_k(n)=\sum\limits_{d|n}d^k$. C. Miguel in \cite{miguel2016menon} extended this identity of Sury from $\Z$ to any residually finite Dedekind domains. A furthur extension of Miguel's result  was given by Li and Kim in \cite[Theorem 1.1]{li2017menon}. For the case of $\Z$, their result  reads as follows \cite[Corollary 1.3]{li2017menon}:
\begin{align}\label{ide:li-kim-identity}
&\sum\limits_{\substack{ a_1,a_2,\ldots,a_s \in U(\Z_n)\\b_1,b_2 \ldots b_r \in \Z_n}}(a_1-1,\ldots,a_s-1,b_1,\ldots b_r,n) \\
&= \phi(n) \prod\limits_{i=1}^{m}(\phi(p_i^{k_i})^{s-1} p_i^{k_ir}-p_i^{k_i(s+r-1)}+\sigma_{s+r-1}(p_i^{k_i}))\nonumber 
\end{align} where $n = p_1^{k_1}p_2^{k_2}\ldots p_m^{k_m}$ is the prime factorization of $n$.  Note that this identity generalizes the classical Menon's identity and Menon-Sury identity.

Various other  generalizations of the Menon's identity were provided by many authors, see for example  \cite{haukkanen2005menon}, \cite{haukkanen1996generalization}, \cite{ramaiah1978arithmetical}, \cite{toth2011menon} and the more recent papers \cite{haukkanen2019menon} and \cite{toth2019short}. 

A natural question arising if one considers a generalization of the usual gcd function (which we define in the next section) in the place of the gcd function appearing in the Menon's identity (\ref{menons-identity}) what could be the possible change that can happen to this identity as well as the other generalizations of it. We propose a very natural generalization of the Li-Kim identity (\ref{ide:li-kim-identity}) involving generalized gcd function and Klee's function in this paper (which in turn generalizes the Menon's identity as well). We prove it using elementary number theory techniques.

 \section{Notations and basic results}
  Most of the notations, functions, and identities we use in this paper are standard and their definitions can be found in \cite{tom1976introduction}. For  a finite set $A$, by $\#A$ we mean the number of elements in $A$.

  The Jordan totient function $J_s(n)$ defined for positive integers $s$ and $n$ gives the number of ordered sets of $s$ elements from a complete residue system (mod $n$) such that the greatest common divisor of each set is prime to $n$ \cite[pp 95-97]{jordan1870traite}.  The  von Sterneck's function $H_s$ is defined as $$H_s(n) = \sum\limits_{\substack{n =[d_1,d_2,\cdots,d_s]}}\phi(d_1)\phi(d_2)\cdots\phi(d_s),$$ where the summation ranges over all ordered set of $s$ positive integers $d_1,d_2,\cdots,d_s$ with their least common multiple equal to $n$. Note that $[a, b]$ denotes the \emph{lcm} of integers $a,b$.

   For $m,n\in\N$, $(m,n)$ will denote the gcd of $m$ and $n$. Generalizing this notion, for positive integer $s$, integers $a,b$, not both zero, the largest $l^s$ (where $l\in \N$) dividing both $a$ and $b$ will be denoted by  $(a,b)_{s}$. Following Cohen \cite{cohen1956some} we call this function on $\N\times \N$ as the generalized gcd function. When $s=1$ this will be equal to the usual gcd function.  Like the gcd function, $(a,b)_s = (b,a)_s$. $a\in \N$ is said to be $s-$power free or $s-$free if no $l^s$ where $l\in \N$  divides $a$.
   
   The Cohen's function $\phi_s$ is defined as follows. If $(a,b)_s=1$, $a,b$ are said to be relatively $s-$prime. The subset $N$ of a complete residue system $M$ (mod $n^s$) consisting of all elements of $M$ that are relatively $s-$prime to $n^s$ is called an $s-$reduced residue system (mod $n$). The number of elements of an $s-$reduced residue system is denoted by $\phi_s(n)$.
   
    The functions $J_s(n)$ and $\phi_s(n)$  are the same \cite[Theorem 5 ]{cohen1956some} though their definitions look different. 
   
  Using the above  generalization of the gcd function,  for positive integers $s$ and $n$ the Klee's function  $\Phi_{s}(n)$ is defined 
   	 to give the cardinality of the  set $\{m\in\N : 1\leq m\leq n, (m,n)_s=1\}$.
   
  Note that $\Phi_1 = \phi$, the usual Euler totient function on $\mathbb{N}$. Some interesting properties of $\Phi_s$ are the following.
\begin{enumerate}
\item For $n,s \in\mathbb{N}$, $\Phi_{s}(n) =  \sum\limits_{d^{s}|n}\ \mu(d)  \frac{n}{d^{s}}$.
\item For $n,s \in\mathbb{N}$, $\Phi_{s}(n) =  n\prod\limits_{\substack{p^s|n\\p\text{ prime}}}(1-\frac{1}{p\textsuperscript{s}})$
where by convention, empty product is taken to be equal to 1.
\item $\Phi_s(p^{a}) =\begin{cases} 
 p^{a}-p^{a-s} \text{ if } a\geq s\\
 p^{a} \text{ otherwise}.
 \end{cases}$, where $p$ is prime and $a \in \mathbb{N}$.
\item $\Phi_s(n)$ is multiplicative in $n$.
\item $\Phi_s(n)$ is not completely multiplicative in $n$. 
 
 \item If $a$ divides $b$ and $(a,\frac{b}{a}) = 1$, then
  $\Phi_s(a)$ divides $\Phi_s(b)$.
 \item  For a prime $p$, $\Phi_s(p) = p$. So $\Phi_s(n)$ need not be even  where as $\phi(n)$ is even for $n>2$.
\item If $2^{s+1}$ divides $n$ or $2^{s-1}\mid n$ and $2^s\nmid n$, then $\Phi_s(n)$ is even.
\item If $p$ is an odd prime such that $p^{s}$ divides $n$, then $\Phi_s(n)$ is even.
\item If $n = 2^sa$, where $a$ is odd and $a$ is $s-$free, then $\Phi_s$ is odd.
\end{enumerate}
Many of the above properties are listed in \cite{klee1948generalization}. The rest can be verified easily via elementary techniques.

   By $\tau_{s}(n)$ where $s,n\in \N$, we mean the number of $l^s$ with $l\in \N$ dividing $n$. The function $\tau_{s}(n)$ is multiplicative in $n$, because for $(m,n)=1$, $\tau_s(mn) = \sum\limits_{\substack{d^s\mid mn}}1 = \sum\limits_{\substack{d_1^s\mid m}}1\sum\limits_{\substack{d_2^s\mid n}}1 = \tau_s(m)\tau_s(n) $. But $\tau_{s}(n)$ is not completely multiplicative as for example $m = p_1^sp_2$ and $n = p_2^{s-1}p_3^s$ gives $\tau_s(mn) \neq \tau_s(m) \tau_s(n).$ 
   The usual sum of divisors function can be generalized as follows:  for $k, s, n\in \N$ define $\sigma_{k,s}(n)$ to be the $k^{th}$ power sum of the $s^{th}$ power divisors of $n$. That is $\sigma_{k,s}(n) = \sum\limits_{\substack{d^s\mid n}}(d^s)^{k}$.
 Note that $\sigma_{k,s}(n) \neq \sigma_{ks}(n) $.
 
The principle of cross-classification   is about counting the  number of elements in  certain sets. Since we use it in our proofs, we state it below.
\begin{theo}\cite[Theorem 5.31]{tom1976introduction}
If $A_1, A_2,\ldots, A_n$ are given subsets of a finite set $A$, then 
\begin{align}\nonumber
\#(A-\bigcup\limits_{i=1}^n A_i ) =& \#A-\sum\limits_{1\leq i\leq n}\#A_i+\sum\limits_{1\leq i<j\leq n}\#(A_i\bigcap A_j)\\
&-\sum\limits_{1\leq i<j<k\leq n}\#(A_i\bigcap A_j\bigcap A_k)+\ldots\nonumber\\
&+(-1)^n\#(A_1\bigcap A_2\bigcap \ldots \bigcap A_n).\nonumber
\end{align}
\end{theo}
 \section{Main results and proofs}
We state below the main results we prove in this paper and provide the proof after that.

As a consequence of the principle of cross-classification, we prove the following.
\begin{theo}\label{count-relprime}
Let $n,d,s,r\in \N, d^s|n$. Let $(r, d^s)_s = 1$. Number of elements in $A=\{r+td^s: t=1,2,\ldots,\frac{n}{d^s}\}$ such that $(r+td^s, n)_s=1$ is $\frac{\Phi_s(n)}{\Phi_s(d^s)}$. 
\end{theo}

\begin{theo}[Generalization of Li-Kim identity (\ref{ide:li-kim-identity})]\label{li-idnty}
Let $m_1,m_2,\cdots,m_k,\\ b_1,b_2,\cdots, b_r,n,s \in \N$ and $a_1,a_2,\cdots, a_k \in \Z$ such that $(a_i,n^s)_s=1$, $i = 1,2,\cdots,k$. Then
\begin{align}
\sum\limits_{\substack{1\leq m_1, m_2,\ldots,m_k\leq n^s\\(m_1,n^s)_s=1\\(m_2,n^s)_s=1\\\cdots \\(m_k,n^s)_s=1\\1\leq b_1, b_2,\ldots,b_r\leq n^s}}(m_1-a_1, m_2-a_2,&\ldots,m_k-a_k,b_1,b_2,\cdots,b_r,n^s)_s\\
&= \Phi_s(n^s)^k\sum\limits_{\substack{d^s|n^s}}\frac{(d^s)^r}{\Phi_s(\frac{n^s}{d^s})^{k-1}}\nonumber
\end{align}
\end{theo}
Note that the Menon-Sury identity (\ref{sury-identity}) is a special case of the above generalization with $k = 1, a_1=1$ and $s=1$.

We proceed to prove our results.  First we prove Theorem \ref{count-relprime}, which is essential in the proof of our generalization.

\begin{proof}[Proof of theorem \ref{count-relprime}]
 This result is a generalization of Theorem 5.32 appearing in \cite{tom1976introduction}. We use the same techniques used                           there to justify our claim.

We have to find the number of elements $r+td^s$ such that $(n,r+td^s)_s=1$. Hence we need to remove elements from $A$ that have $(r+td^s,n)_s>1$. If for an element $r+td^s$ of $A$, $p^s|n$ and $p^s|r+td^s$, then since $(r,n)_s=1$, $p^s\nmid d^s$. 
 Therefore the number we require is the number of elements in $A$ with $p^s|n$ and $p^s\nmid d^s$ for some prime $p$. Let these primes be $p_1,p_2,\ldots, p_m$. Write $l = p_1^sp_2^s\ldots p_m^s$.
 Let $A_{i} = \{x: x \in A$ and $p_{i}^{s}|x\}$, $i=1, 2,\cdots, m $. If $x\in A_{i}$ and $x=r+td^s$, then $r+td^s\equiv 0$ (mod $p_{i}^{s}$). This means that  $td^s\equiv -r(\text{mod }p_{i}^{s}$). Since $p_{i}^s\nmid d^s$ (which is if and only if $p_i \nmid d$), there is a unique $t$ mod $p_{i}^{s}$ satisfying this congruence equation. Therefore there exists exactly one $t$ in each of the intervals $[1, p_{i}^{s}], [p_{i}^{s}+1, 2p_{i}^{s}], \cdots, [(q-1)p_{i}^{s}+1, qp_{i}^{s}]$ where $qp_{i}^s = \frac{n}{d^s}$. 
 Therefore, $\#(A_{i})= q= \frac{n/d^s}{p_{i}^{s}}$.
 Similarly, $\#(A_{i}\bigcap A_{j})= \frac{n/d^s}{p_{i}^{s}p_{j}^{s}}$,
 $\cdots$,  $\#(A_{1}\bigcap A_{2}\bigcap \cdots \bigcap A_{m})= \frac{n/d^s}{p_{1}^{s}p_{2}^{s}\cdots p_{m}^{s}}$.

 Hence by cross classification principle, the number of elements we seek is equal to
\begin{align*}
 \#(A-\bigcup \limits _{i=1}^{m}A_{i}) =& \#(A)- \sum \limits_{i=1}^{m}\#(A_{i}) + \sum \limits_{1\leq i<j\leq m }\#(A_{i}\bigcap A_{j})- \cdots \\
 & + (-1)^m \#(A_{1}\bigcap A_{2} \bigcap \cdots \bigcap A_{m})\\
 &= \frac{n}{d^s} - \sum \frac{n/d^s}{p_{i}^{s}} + \sum \frac{n/d^s}{p_{i}^{s}p_{j}^{s}} + \cdots + (-1)^m \frac{n/d^s}{p_{1}^{s}p_{2}^{s}\cdots p_{m}^{s}}\\
 &= \frac{n}{d^s}\left(1 - \sum \frac{1}{p_{i}^{s}} + \sum \frac{1}{p_{i}^{s}p_{j}^{s}} + \cdots +   \frac{(-1)^m}{p_{1}^{s}p_{2}^{s}\cdots p_{m}^{s}}\right)\\
 &= \frac{n}{d^s} (1-\frac{1}{p_{1}^{s}})(1-\frac{1}{p_{2}^{s}})\cdots(1-\frac{1}{p_{m}^{s}}) \\
 &= \frac{n}{d^s} \prod \limits_{p^s|l}(1- \frac{1}{p^{s}})\\
 &= \frac{n}{d^s} \frac{\prod \limits_{p^s|n}(1- \frac{1}{p^{s}})}{\prod \limits_{p^s|d^s}(1- \frac{1}{p^{s}})}\\
 &= \frac{\Phi_s(n)}{\Phi_s(d^s)}.
 \end{align*}

\end{proof}

Next we prove the generalization we proposed.  Here we use elementary number theoretic techniques in the proof. Li and Kim used direct computations involving Dedekind domains to derive their identity. 

\begin{proof}[Proof of theorem \ref{li-idnty}]
We know that \cite[Section V.3]{sivaramakrishnan1988classical}  $n^s =  \sum\limits_{d|n}J_s(d)$ and $J_s(n) = \Phi_s(n^s)$. Hence $n^s =  \sum\limits_{d|n}\Phi_s(d^s)=\sum\limits_{d^s|n^s}\Phi_s(d^s)$. Now $(m_1-a_1, m_2-a_2,\ldots,m_k-a_k,b_1,b_2,\cdots,b_r,n^s)_s$ is an $s^{\text{th}}$ power of some natural number. So
\begin{align*}
&\sum\limits_{\substack{1\leq m_1, m_2,\ldots,m_k\leq n^s\\(m_1,n^s)_s=1\\(m_2,n^s)_s=1\\\cdots \\(m_k,n^s)_s=1\\1\leq b_1, b_2,\ldots,b_r\leq n^s}}(m_1-a_1, m_2-a_2,\ldots,m_k-a_k,b_1,b_2,\cdots,b_r,n^s)_s\\
&=\sum\limits_{\substack{m_1=1\\(m_1,n^s)_s = 1}}^{n^s}
	 \cdots \sum\limits_{\substack{m_k=1\\(m_k,n^s)_s = 1}}^{n^s}\sum\limits_{\substack{b_1=1}}^{n^s} \cdots \sum\limits_{\substack{b_r=1}}^{n^s}(m_1-a_1, m_2-a_2,\ldots,m_k-a_k,b_1,b_2,\cdots,b_r,n^s)_s\\
	 &=\sum\limits_{\substack{m_1=1\\(m_1,n^s)_s = 1}}^{n^s}
	 \cdots \sum\limits_{\substack{m_k=1\\(m_k,n)_s = 1}}^{n^s}\sum\limits_{\substack{b_1=1}}^{n^s} \cdots \sum\limits_{\substack{b_r=1}}^{n^s} \quad\sum\limits_{\substack{d^s \mid (m_1-a_1, m_2-a_2,\ldots,m_k-a_k,b_1,b_2,\cdots,b_r,n^s)_s}}\Phi_s(d^s)\\ 
	 & = \sum\limits_{\substack{d^S\mid n^s}}\Phi_s(d^s)\sum\limits_{\substack{b_1=1\\d^s\mid b_1}}^{n^s} \cdots \sum\limits_{\substack{b_r=1\\d^s\mid b_r}}^{n^s} \sum\limits_{\substack{m_1=1\\(m_1,n)_s = 1\\m_1\equiv 1(\text{mod }d^s)}}^{n^s}\cdots \sum\limits_{\substack{m_k=1\\(m_k,n^s) = 1\\m_k\equiv 1(\text{mod }d^s)}}^{n^s}1\\
	 & =  \sum\limits_{\substack{d^s\mid n^s}}\Phi_s(d^s)\sum\limits_{\substack{b_1=1\\d^s\mid b_1}}^{n^s} \cdots \sum\limits_{\substack{b_r=1\\d^s\mid b_r}}^{n^s}\Big(\frac{\Phi_s(n^s)}{\Phi_s(d^s)}\Big)^k\text{    (using Theorem \ref{count-relprime})}\\ 
	 & = \Phi_s(n^s)^k\sum\limits_{\substack{d^s\mid n^s}}\frac{1}{\Phi_s(d^s)^{k-1}}\sum\limits_{\substack{b_1=1\\d^s\mid b_1}}^{n^s} \cdots \sum\limits_{\substack{b_r=1\\d^s\mid b_r}}^{n^s}1\\& = \Phi_s(n^s)^k \sum\limits_{\substack{d^s\mid n^s}}\frac{1}{\Phi_s(d^s)^{k-1}}(\frac{n^s}{d^s})^{r}\\& = \Phi_s(n^s)^k \sum\limits_{\substack{d^s\mid n^s}}\frac{(d^s)^r}{\Phi_s(\frac{n^s}{d^s})^{k-1}},
\end{align*}
which completes the proof.
\end{proof}
We will now show that the above identity is indeed the same as  Li-Kim identity (\ref{ide:li-kim-identity}) when $s= a_i = 1$. For that, we need to show that the RHS of identities $(4)$ and (\ref{ide:li-kim-identity}) are equal as the LHS of our identity can be quickly seen to be equal to the LHS of (\ref{ide:li-kim-identity})  when $s= a_i = 1$. To prove that the RHS are also equal, we require the following identity.
\begin{lemm}
\begin{align}
\sum\limits_{\substack{d \mid  p^{t}}}d^r  \Big(\frac{\phi( p^{t})}{\phi(\frac{ p^{t}}{d})}\Big)^{k}&= \sum\limits_{\substack{j = 0}}^{t-1} p^{j(k+r)}+p^{t(k+r)}(1-\frac{1}{p})^k\nonumber
\end{align}
\begin{proof}
\begin{align*}
\sum\limits_{\substack{d \mid  p^{t}}}d^r  \Big(\frac{\phi( p^{t})}{\phi(\frac{ p^{t}}{d})}\Big)^{k}&
=\sum\limits_{\substack{j = 0}}^{t} (p^{j})^{r}\Big(\frac{\phi(p^t)}{\phi(p^{t-j})}\Big)^k\\
&=\sum\limits_{\substack{j = 0}}^{t-1} p^{jr}\Big(\frac{p^t(1-\frac{1}{p})}{p^{t-j}(1-\frac{1}{p})}\Big)^k+p^{tr}p^{tk}(1-\frac{1}{p})^k\\
&=\sum\limits_{\substack{j = 0}}^{t-1} p^{j(k+r)}+p^{t(k+r)}(1-\frac{1}{p})^k.
\end{align*}
\end{proof}
\end{lemm}
Now we show what we claimed; that is 
$$\sum\limits_{\substack{d \mid n}}d^r  \Big(\frac{\phi(n)}{\phi(\frac{n}{d})}\Big)^{k-1} = \prod\limits_{i=1}^{q}\Big(\phi(p_i^{t_i})^{k-1} p_i^{t_ir}-p_i^{t_i(k+r-1)}+\sigma_{k+r-1}(p_i^{t_i})\Big)$$ where $n = p_1^{t_1}p_2^{t_2}\ldots p_q^{t_q}$.

Starting from the RHS, we have
\begin{align*}
&\prod\limits_{i=1}^{q}\Big(\phi(p_i^{t_i})^{k-1} p_i^{t_ir}-p_i^{t_i(k+r-1)}+\sigma_{k+r-1}(p_i^{t_i})\Big)\\&
 = \prod\limits_{i=1}^{q}\Big({p_i^{{t_i}(k-1)}} (1-\frac{1}{p_i})^{k-1}p_i^{t_ir}-p_i^{t_i(k+r-1)}+1+p_i^{k+r-1}+{p_i^{2(k+r-1)}}\\&+\ldots+{p_i^{(t_i-1)(k+r-1)}} +{p_i^{{t_i}(k+r-1)}}\Big)\\& 
 =\prod\limits_{i=1}^{q}\Big(1+p_i^{k+r-1}+{{p_i^{2(k+r-1)}}}+\ldots +{p_i^{{(t_i-1)}(k+r-1)}}+{p_i^{{t_i}(k+r-1)}} (1-\frac{1}{p_i})^{k-1}\Big)\\& 
 =\prod\limits_{i=1}^{q}\Big(\sum\limits_{\substack{j = 0}}^{t_i-1}{p_i^{{j}(k+r-1)}}+{p_i^{{t_i}(k+r-1)}}(1-\frac{1}{p_i})^{k-1}\Big)
 \\&= \prod\limits_{i=1}^{q}\sum\limits_{\substack{d_i \mid p_i^{t_i}}}d_i^r  \Big(\frac{\phi(p_i^{t_i})}{\phi(\frac{p_i^{t_i}}{d_i})}\Big)^{k-1}  \text{ (by identity (5))}\\&
 = \sum\limits_{\substack{d_1 \mid p_1^{t_1}}}d_1^r  \Big(\frac{\phi(p_1^{t_1})}{\phi(\frac{p_1^{t_1}}{d_1})}\Big)^{k-1} \sum\limits_{\substack{d_2 \mid p_2^{t_2}}}d_2^r  \Big(\frac{\phi(p_2^{t_2})}{\phi(\frac{p_2^{t_2}}{d_2})}\Big)^{k-1} \ldots \sum\limits_{\substack{d_q \mid p_q^{t_q}}}d_q^r  \Big(\frac{\phi(p_q^{t_q})}{\phi(\frac{p_q^{t_q}}{d_q})}\Big)^{k-1} \\&
 = \sum\limits_{\substack{d_1d_2\ldots d_q\mid  p_1^{t_1}p_2^{t_2}\ldots p_q^{t_q}}}(d_1d_2\ldots d_q)^r  \Big(\frac{\phi( p_1^{t_1}p_2^{t_2}\ldots p_q^{t_q})}{\phi(\frac{ p_1^{t_1}p_2^{t_2}\ldots p_q^{t_q}}{d_1d_2\ldots d_q})}\Big)^{k-1} \\&
 =  \sum\limits_{\substack{d \mid n}}d^r  \Big(\frac{\phi(n)}{\phi(\frac{n}{d})}\Big)^{k-1}.
\end{align*}
Therefore we  get Li-Kim identity (\ref{ide:li-kim-identity}) as a special case of our identity $(4)$ when $s=1$ and $a_i = 1$. Hence our identity also gives a generalization of the Menon-Sury identity which in turn is a generalization of the Menon's identity.

The following identities can be easily deduced from our result by giving special values to $a_i, b_i$ and $s_i$ and may be of independent interest. Note that the first one gives another generalization of the Li-Kim identity and it involves the usual gcd function.
\begin{coro}

\begin{enumerate}
\item 
\begin{displaymath}
 \sum\limits_{\substack{1\leq m_1, m_2,\ldots,m_k\leq n\\(m_1,n)=1\\(m_2,n)=1\\\ldots \\(m_k,n)=1\\1\leq b_1, b_2,\ldots,b_r\leq n}}(m_1-a_1, m_2-a_2,\ldots,m_k-a_k,b_1,b_2,\cdots,b_r,n)\\ = \phi(n)^k \sum\limits_{\substack{d\mid n}}\frac{(d)^r}{\phi(\frac{n}{d})^{k-1}}
 \end{displaymath}
\item 
\begin{displaymath}
\sum\limits_{\substack{1\leq m_1, m_2,\ldots,m_k\leq n^s\\(m_1,n^s)_s=1\\(m_2,n^s)_s=1\\\cdots \\(m_k,n^s)_s=1}}(m_1-a_1, m_2-a_2,\ldots,m_k-a_k,n^s)_s = \Phi_s(n^s)^k \sum\limits_{\substack{d^s\mid n^s}}\frac{1}{\Phi_s(d^s)^{k-1}}
 \end{displaymath}

\item \begin{displaymath}
\sum\limits_{\substack{m=1\\(m.n^s)_s=1}}^{n^s} (m-1,n^s)_s=\Phi_s(n^s)\tau_s(n^s)
 \end{displaymath}
\end{enumerate}
\end{coro}
\section{An alternating way of defining $\Phi_s$ and extending it further}

In \cite{tarna2015} M. T\u{a}rn\u{a}uceanu suggested a new generalization of $\phi$  using elementary concepts in group theory. His generalization was based on the following idea. An element $m \in \Z_n$ is a generator of the group $(\Z_n,+) $ if and only if  $(m,n) = 1$ which is if and only if $o(m) = n = exp(\Z_n)$, where $o(m)$ is the order of the element $m$ and $exp(\Z_n)$ is the exponent of the group $(\Z_n,+)$. Thus $\phi(n)$ is the number of elements of order $n$ in $\Z_n$. That is $\phi(n) = \#\{m\in \Z_n \mid o(m) = exp(\Z_n)\}$. T\u{a}rn\u{a}uceanu extended $\phi$ to an arbitrary finite group $G$ by defining  $\phi(G) = \# \{m\in G \mid o(m) = exp(G)\}$.

 We may adapt this technique for defining the generalization $\Phi_s$ as follows. An $m \in \N$ can be counted in $\Phi_s(n)$ if and only if $1\leq m\leq n$ and $(m,n)_s = 1$. Now $o(m) = \frac{n}{(m,n)} $, and $(m,n)_s = 1$ if and only if $m$ and $n$ do not share any prime factor with power greater than or equal to $s$. That is $(m,n) = p_1^{a_1}p_2^{a_2}\cdots p_r^{a_r}$ with $0 \leq a_i<s$. Here $p_i$ are prime divisors of $n$. Therefore $o(m) = \frac{n}{(m,n)} = \frac{n}{p_1^{a_1}p_2^{a_2}\cdots p_r^{a_r}}$, $0 \leq a_i<s$. By using this fact, we may observe that $\Phi_s(n) = \#\{m\in \Z_n \mid o(m)= \frac{n}{p_1^{a_1}p_2^{a_2}\cdots p_r^{a_r}},p_i^{a_i}\mid n \text{ and }0 \leq a_i<s, i = 1,2,\cdots r \}$. 

Now the extension of $\Phi_s$ can be defined as the following. For any arbitrary finite group $G$, define $\Phi_s(G) = \#\{a\in G \mid o(a) = \frac{exp(G)}{p_1^{a_1}p_2^{a_2}\cdots p_r^{a_r}}, p_i^{a_i}\mid exp(G) \text{ and } 0\leq a_i< s, i = 1,2,\cdots r\}$. With this definition we have the following quick observations. For any finite cyclic group $G$, $\Phi_s(G) = \Phi_s(\#G)$. For relatively prime integers $m$ and $n$, we have $\Phi_s(\Z_m\times \Z_n) = \Phi_s(\Z_{mn})= \Phi_s(mn) $. For $s-$free integers $m$ and $n$, $\Phi_s(\Z_m\times \Z_n) = \Phi_s(\Z_{m})\Phi_s(\Z_{n})$. The last statement follows because
\begin{align*}
\Phi_s(\Z_m\times \Z_n) &= \#\{a\in \Z_m\times \Z_n \mid o(a) = \frac{exp(\Z_m\times \Z_n)}{p_1^{a_1}p_2^{a_2}\cdots p_k^{a_k}},\\&\text{ }p_i^{a_i}\mid exp(\Z_m\times \Z_n) \text{ and } 0\leq a_i<s, i = 1,2,\cdots k \}\\&= \#\{a\in \Z_m\times \Z_n \mid o(a) = \frac{lcm(m,n)}{p_1^{a_1}p_2^{a_2}\cdots p_k^{a_k}},\\&\text{ }p_i^{a_i}\mid lcm(m,n) \text{ and } 0\leq a_i<s, i = 1,2,\cdots k \}
\end{align*}
Note that if $m$ and $n$ are $s-$free then $lcm(m,n)$ is also an $s-$free integer. We have $\Phi_s(\Z_m\times \Z_n) = \#\{a\in \Z_m\times \Z_n : o(a) = d, \text{ where }d\mid lcm(m,n) \} = mn = \Phi_s(\Z_m)\times \Phi_s(\Z_n)$. It is not very difficult to deduce the following general statement. For $s-$free integers $m_1,m_2,\cdots ,m_k$, $\Phi_s(\Z_{m_1}\times \Z_{m_2}\times \Z_{m_k}) = \Phi_s(\Z_{m_1})\Phi_s(\Z_{m_2})\cdots \Phi_s(\Z_{m_k})$.

\section{Further directions}
Since we feel that this is the first time Menon's identity is revisited through the generalized gcd concept, it would be interesting to see what possible results can be obtained if one tries to apply our techniques to other generalizations of the identity. We note that we have investivated how does the identity of Zhao and Cao in \cite{zhao2017another} change if one uses the generalized gcd, $\Phi_s$ and $\tau_s$ in a recent (unpublished) work. We expect our generalization to have interesting consequences in group theory considering the definition of $\Phi_s$ we gave in the previous section.

\section{Acknowledgements}
We would like to thank the reviewer for suggesting the possible generalization of Li-Kim identity before which our attention was restricted to giving a generalization of the Menon-Sury identity alone.

The first author thanks the University Grants Commission of India for providing financial support for carrying out research work through their Junior Research Fellowship (JRF) scheme. The second author thanks the Kerala State Council for Science,Technology and Environment, Thiruvananthapuram, Kerala, India for providing financial support  for carrying out research work.


\begin{thebibliography}{10}

\bibitem{tom1976introduction}
Tom Apostol.
\newblock {\em Introduction to analytic number theory}.
\newblock Springer, 1976.

\bibitem{cohen1956some}
Eckford Cohen.
\newblock Some totient functions.
\newblock {\em Duke Mathematical Journal}, 23(4):515--522, 1956.

\bibitem{haukkanen2005menon}
Pentti Haukkanen.
\newblock Menon’s identity with respect to a generalized divisibility
  relation.
\newblock {\em aequationes mathematicae}, 70(3):240--246, 2005.

\bibitem{haukkanen2019menon}
Pentti Haukkanen and L{\'a}szl{\'o} T{\'o}th.
\newblock Menon-type identities again: Note on a paper by li, kim and qiao.
\newblock {\em Publicationes Mathematicae: Debrecen}, 96:487--502, 2020.

\bibitem{haukkanen1996generalization}
Pentti Haukkanen and Jun Wang.
\newblock A generalization of menon's identity with respect to a set of
  polynomials.
\newblock {\em Portugaliae Mathematica}, 53(3):331--338, 1996.

\bibitem{jordan1870traite}
Camille Jordan.
\newblock {\em Traite des substitutions et des equations algebriques par m.
  Camille Jordan}.
\newblock Gauthier-Villars, 1870.

\bibitem{klee1948generalization}
VL~Klee.
\newblock A generalization of euler's $\varphi$-function.
\newblock {\em The American Mathematical Monthly}, 55(6):358--359, 1948.

\bibitem{li2017menon}
Yan Li and Daeyeoul Kim.
\newblock A menon-type identity with many tuples of group of units in
  residually finite dedekind domains.
\newblock {\em Journal of Number Theory}, 175:42--50, 2017.

\bibitem{menon1965sum}
P~Kesava Menon.
\newblock On the sum $\sum (a-1, n),[(a, n)= 1]$.
\newblock {\em The Journal of the Indian Mathematical Society}, 29(3):155--163,
  1965.

\bibitem{miguel2016menon}
C~Miguel.
\newblock A menon-type identity in residually finite dedekind domains.
\newblock {\em Journal of Number Theory}, 164:43--51, 2016.

\bibitem{ramaiah1978arithmetical}
V~Sita Ramaiah.
\newblock Arithmetical sums in regular convolutions.
\newblock {\em J. Reine Angew. Math}, 303(304):265--283, 1978.

\bibitem{rao1972certain}
K~Nageswara Rao.
\newblock On certain arithmetical sums.
\newblock In {\em The Theory of Arithmetic Functions}, pages 181--192.
  Springer, 1972.

\bibitem{sivaramakrishnan1988classical}
R~Sivaramakrishnan.
\newblock {\em Classical theory of arithmetic functions}, volume 126.
\newblock CRC Press, 1988.

\bibitem{sury2009some}
Balasubramanian Sury.
\newblock Some number-theoretic identities from group actions.
\newblock {\em Rendiconti del circolo matematico di Palermo}, 58(1):99--108,
  2009.

\bibitem{toth2011menon}
L{\'a}szl{\'o} T{\'o}th.
\newblock Menon's identity and arithmetical sums representing functions of
  several variables.
\newblock {\em Rend. Sem. Mat. Univ. Politec. Torino}, 69(1):97--110, 2011.

\bibitem{toth2018menon}
L{\'a}szl{\'o} T{\'o}th.
\newblock Menon-type identities concerning dirichlet characters.
\newblock {\em International Journal of Number Theory}, 14(04):1047--1054,
  2018.

\bibitem{toth2019short}
L{\'a}szl{\'o} T{\'o}th.
\newblock Short proof and generalization of a menon-type identity by li, hu and
  kim.
\newblock {\em Taiwanese Journal of Mathematics}, 23(3):557--561, 2019.

\bibitem{tarna2015}
Marius T\u{a}rn\u{a}uceanu.
\newblock A generalization of the euler's totient function.
\newblock {\em Asian-European Journal of Mathematics}, 8(4), 2015.

\bibitem{zhao2017another}
Xiao-Peng Zhao and Zhen-Fu Cao.
\newblock Another generalization of menon’s identity.
\newblock {\em International Journal of Number Theory}, 13(9):2373--2379, 2017.

\end{thebibliography}

\end{document}